\documentclass[12pt]{amsart}
\usepackage{latexsym,amssymb,amscd}

\def\NZQ{\Bbb}               % the font for N,Z,Q,R,C
\def\NN{{\NZQ N}}
\def\QQ{{\NZQ Q}}
\def\ZZ{{\NZQ Z}}

\def\B'c{{\mathcal{B'}}}
\def\U'c{{\mathcal{U'}}}

\def\opn#1#2{\def#1{\operatorname{#2}}} % to make operators
\opn\chara{char}
\opn\length{\ell}
\opn\projdim{proj\,dim}
\opn\injdim{inj\,dim}
\opn\ini{in}
\opn\rank{rank}
\opn\depth{depth}
\opn\sdepth{sdepth}
\opn\height{ht}
\opn\embdim{emb\,dim}
\opn\codim{codim}

\opn\Tr{Tr}
\opn\bigrank{big\,rank}
\opn\superheight{superheight}\opn\lcm{lcm}
\opn\trdeg{tr\,deg}%
\opn\reg{reg}
\opn\lreg{lreg}
\opn\set{set}
\opn\supp{Supp}
\opn\shad{Shad}
\opn\div{div}
\opn\Div{Div}
\opn\cl{cl}
\opn\Cl{Cl}
\opn\Spec{Spec}
\opn\Supp{Supp}
\opn\supp{supp}
\opn\Sing{Sing}
\opn\Ass{Ass}
\opn\Min{Min}
\opn\Ann{Ann}
\opn\Rad{Rad}
\opn\Soc{Soc}
\opn\Ker{Ker}
\opn\Coker{Coker}
\opn\Im{Im}
\opn\Hom{Hom}
\opn\Tor{Tor}
\opn\Ext{Ext}
\opn\End{End}
\opn\Aut{Aut}
\opn\id{id}

\opn\nat{nat}
\opn\GL{GL}
\opn\SL{SL}
\opn\mod{mod}
\opn\ord{ord}
\opn\aff{aff}
\opn\con{conv}
\opn\relint{relint}
\opn\st{st}
\opn\lk{lk}
\opn\cn{cn}
\opn\core{core}
\opn\vol{vol}
\opn\gr{gr}

\def\pot#1#2{#1[\kern-0.28ex[#2]\kern-0.28ex]}

\opn\dirlim{\underrightarrow{\lim}}
\opn\invlim{\underleftarrow{\lim}}
\def\pnt{{\raise0.5mm\hbox{\large\bf.}}}

\def\Implies{\ifmmode\Longrightarrow \else
     \unskip${}\Longrightarrow{}$\ignorespaces\fi}
\def\implies{\ifmmode\Rightarrow \else
     \unskip${}\Rightarrow{}$\ignorespaces\fi}
\def\iff{\ifmmode\Longleftrightarrow \else
     \unskip${}\Longleftrightarrow{}$\ignorespaces\fi}

\let\:=\colon
\newtheorem{Theorem}{Theorem}[section]
\newtheorem{Lemma}[Theorem]{Lemma}
\newtheorem{Corollary}[Theorem]{Corollary}
\newtheorem{Proposition}[Theorem]{Proposition}
\newtheorem{Remark}[Theorem]{Remark}

\newtheorem{Example}[Theorem]{Example}

\newtheorem{Question}[Theorem]{Question}
\let\epsilon=\varepsilon
\let\phi=\varphi
\let\kappa=\varkappa
\numberwithin{equation}{section}

\title{Stanley depth and complete $k$-partite hypergraphs}

\author[Muhammad Ishaq]{Muhammad Ishaq}
\address{Muhammad Ishaq, Abdus Salam School of Mathematical Sciences, GC
University, Lahore, 68-B New Muslim town Lahore, Pakistan.}
\email{ishaq$\_\,$maths@yahoo.com}

\author[Muhammad Imran Qureshi]{Muhammad Imran Qureshi}
\address{Muhammad Imran Qureshi, Comsats Institute of Information Technology, Vehari campus, Multan road Vehari, Pakistan.}
\email{imranqureshi18@gmail.com}
\thanks{The authors would like to express their gratitude to ASSMS of GC University Lahore for creating a very appropriate atmosphere for research work. This research is partially supported by HEC Pakistan.}
\begin{document}
\maketitle
\begin{abstract}
We give an upper bound for the Stanley depth of the edge ideal of a complete $k$-partite hypergraph and as an application we give an upper bound for the Stanley depth of a monomial ideal in a polynomial ring $S$. We also give a lower and an upper bound for the cyclic module $S/I$ associated to the complete $k$-partite hypergraph.\\\\
Key words: Monomial ideals, Stanley decomposition, Stanley depth.\\
2000 Mathematics Subject Classification: Primary 13C15, Secondary 13P10, 13F20, 05E45, 05C65.
\end{abstract}
\section{Introduction}
Let $K$ be a field, $S = K[x_1,\dots,x_n]$ be the polynomial ring in $n$ variables over a field $K$ and $M$ be a finitely generated $\ZZ^n$-graded $S$-module. Let $u\in M$ be a homogeneous element in $M$ and $Z$ a subset of the set of variables $Z\subset\{x_1,\dots,x_n\}$. We denote by $uK[Z]$ the $K$-subspace of $M$ generated by all elements $uv$ where $v$ is a monomial in $K[Z]$. If $uK[Z]$ is a free $K[Z]$-module, the $\ZZ^n$-graded $K$-space $uK[Z]\subset M$ is called a Stanley space of dimension $|Z|$.
A Stanley decomposition of M is a presentation of the $\ZZ^n$-graded $K$-vector space $M$ as a finite direct sum of Stanley spaces $$\mathcal{D}:M=\bigoplus\limits_{i=1}^{s}u_iK[Z_i].$$
The number
$$\sdepth\mathcal{D}=\min\{|Z_i| : i =1,\dots,s\}$$ is called the Stanley depth of decomposition $\mathcal{D}$ and the number
$$\sdepth M := \max\{\sdepth\mathcal{D} : \mathcal{D} \hbox{ \ is a Stanley decomposition of $M$}\}$$
is called the Stanley depth of M. This is a combinatorial invariant and does not depend on the characteristic of $K$.
The following open conjecture is due to Stanley \cite{RP}:
$$\depth M\leq \sdepth M,$$ for all finitely generated $\ZZ^n$-graded $S$-modules $M$.\\
\indent Let  $\mathbf{H}=(V,E)$ denote a hypergraph with vertex set $V$ and hyperedge set $E$. A hyperedge $e\in E$ is a subset of the vertices. That is, $e\subset V$ for each $e\in E$. A hypergraph is called complete $k$-partite if the vertices are partitioned into $k$ disjoint subsets $V_i$, $i= 1,\dots,k$ and $E$ consists of all hyperedges containing exactly one vertex from each of the $k$ subsets.\\
\indent  In this paper we try to answer the following question asked by B. Nill and K. Vorwerk in \cite{BK}.
\begin{Question}[\cite{BK}]{\em
Let $I$ be the edge ideal of a complete $k$-partite hypergraph $\mathbf{H}_d^k$. Here, $\mathbf{H}_d^k$ has $kd$ vertices divided into $k$ independent sets $V^{(i)}$ (for $i=1,\dots,k$) each with $d$ vertices $v_1^{(i)},\dots,v_d^{(i)}$, and $\mathbf{H}_d^k$ has $d^k$ hyperedges consisting of exactly $k$ vertices. Then $I$ is squarefree monomial ideal in the polynomial ring $K[v_j^{(i)}:i\in \{1,\dots,k\},j\in \{1,\dots,d\}]$:
$$I=(v_1^{(1)},\dots,v_d^{(1)})\cdots (v_1^{(k)},\dots,v_d^{(k)}).$$
What is $\sdepth(S/I)$ in this case?
}
\end{Question}
We consider this question even in more general frame. We consider the case where each vertex set $V^{(i)}$ is not necessarily of the same cardinality. Let $I$ be the edge ideal of a complete $k$-partite hypergraph $\mathbf{H}^k$, where $\mathbf{H}^k$ has $n$ vertices divided into $k$ independent sets $V^{(i)}$ (for $i=1,\dots,k$) each with $d_i$ vertices $v_1^{(i)},\dots,v_{d_i}^{(i)}$, and $\mathbf{H}^k$ has $d_1d_2\cdots d_k$ hyperedges consisting of exactly $k$ vertices. To each vertex set $V^{(i)}$ we associate a set of variables $\{x_{i_1},\dots,x_{i_{d_i}}\}$ and set $S=K[(x_{i_j})]$. Now let $V^{(i)}$ and $V^{(j)}$ be two vertex sets, $\{x_{i_1},\dots,x_{i_{d_i}}\}$ and $\{x_{j_1},\dots,x_{j_{d_j}}\}$ be the sets of variables associated to $V^{(i)}$ and $V^{(j)}$ respectively. Since $V^{(i)}$ and $V^{(j)}$ are independent we have $\{x_{i_1},\dots,x_{i_{d_i}}\}\cap\{x_{j_1},\dots,x_{j_{d_j}}\}=\emptyset$. Then $I$ is the squarefree monomial ideal in the polynomial ring $S$: $$I=P_1P_2\cdots P_k=P_1\cap P_2\cap \dots \cap P_k,$$
where $P_i=(x_{i_1},\dots,x_{i_{d_i}})$ and $\sum\limits_{i=1}^kP_i=\mathfrak{m}=(x_1,\dots,x_n)$. We give a tight upper bound to $\sdepth(I)$ see our Theorem \ref{intersections} and as an application we give an upper bound for the Stanley depth of a monomial ideal (see Theorem \ref{General}). In some cases we are able to give the exact values of $\sdepth(I)$ (see our corollaries \ref{cor3}, \ref{cor5}). We also give a tight lower bound to $\sdepth(S/I)$ see our Theorem \ref{thm}. Our Proposition \ref{prop} gives an upper bound for Stanley depth of $S/I$ but is big in general. \\ We owe thanks to the Referee who inspired us Theorem \ref{th2}.

\section{Edge ideals of a complete $k$-partite hypergraph}
We recall the method of Herzog et al. \cite{HVZ} for computing the Stanley depth of a monomial ideal $I$ using posets. For $c\in \NN^n$, let $x^c$ denote the monomial $x_1^{c(1)}x_2^{c(2)}\cdots x_n^{c(n)}$ and let $I=(x^{a_1},x^{a_2},\dots,x^{a_m})$ be a monomial ideal of $S$. Let $g\in \NN^n$ be the componentwise maximum of $a_i$. Then characteristic poset of $I$ with respect to $g$ (see \cite{HVZ}), denoted by $\mathcal{P}_I^{g}$ is in fact the set $$\mathcal{P}_I^{g}=\{c\in \NN^n \ | \ c\leq g, \text{there is $i$ such that $c\geq a_i$}\},$$ where $\leq$ denotes the partial order in $\NN^n$ which is given by componentwise comparison. For every $a,b\in \mathcal{P}_I^{g}$ with $a\leq b$, define the interval $[a,b]$ to be $\{c\in \mathcal{P}_I^{g}:a\leq c\leq b\}$. Let $\mathcal{P}:\mathcal{P}_I^{g}=\cup_{i=1}^r[c_i,d_i]$ be a partition of $\mathcal{P}_I^{g}$, define $\rho(d_i):=|\{j: d_i(j)=g(j)\}|.$ Define the Stanley depth of a partition $\mathcal{P}$ to be $$\sdepth (\mathcal{P})=\min\limits_{[c_i,d_i]\in \mathcal{P}} \rho(d_i)$$ and the Stanley depth of the poset $\mathcal{P}_I^g$ to be $\sdepth(\mathcal{P}_I^g):=\max\limits_{\mathcal{P}}\,\sdepth(\mathcal{P})$, where maximum is taken over all the partitions $\mathcal{P}$ of $\mathcal{P}_I^g$ into intervals. Herzog et al., showed in \cite{HVZ} that $\sdepth(I)=\sdepth(\mathcal{P}_I^g).$ Next lemma is a small extension of \cite[Lemma 1.1]{C}, its proof is given for the sake of our completeness.
\begin{Lemma}\label{lemma1}
Let $r,m$ and $a$ be positive integers with $r<m$ and $v_1,\dots,v_m \in K[x_2,\dots,x_n]$ be some monomials of $S$. Let $I=(x_1^{a}v_1,\dots,x_1^{a}v_r,v_{r+1},\dots,v_{m})$ and $I'=(x_1^{a+1}v_1,\dots, x_1^{a+1}v_r,v_{r+1},\dots,v_m)$ be monomial ideals of $S$. Then $$\sdepth(I)=\sdepth(I').$$
\end{Lemma}
\begin{proof}
Let $\mathcal{P}:\mathcal{P}_I^g=\bigcup_{i=1}^r[c_i,d_i]$ be a partition of $\mathcal{P}_I^g$ such that $\sdepth(\mathcal{P})=\sdepth(I)$, Define $$c_{i}'=\left\{
                                              \begin{array}{ll}
                                                c_i, & \hbox{if $c_i(1)<a$} \\
                                                c_i+e_1, & \hbox{if $c_i(1)=a$}
                                              \end{array}
                                            \right.
\ and \ \ d_i'=\left\{
      \begin{array}{ll}
        d_i, & \hbox{if $d_i(1)<a-1$} \\
        d_i+e_1, & \hbox{if $d_i(1)\geq a-1$}
      \end{array}
    \right.
.$$
Let $g':=(g(1)+1,g(2),\dots,g(n))\in \NN^n.$ Let $\mathcal{P}_{I'}^{g'}$ be the characteristic poset of $I'$ with respect to $g'$. We claim that there exists a partition $\mathcal{P}':\mathcal{P}_{I'}^{g'}=\bigcup_{i=1}^r[c_i',d_i']$ of $\mathcal{P}_{I'}^{g'}$. Note that $\rho(d_i)=\rho(d_i')$ for all $i\in [r].$ Indeed, $d_i'(1)=g'(1)=a+1$ if and only if $d_i(1)=g(1)=a$ and $d_i'(j)=g(j)$ if and only if $d_i(j)=g(j)$, for all $j\geq 2$. Therefore $\sdepth(I)\leq \sdepth(I')$. Now we have to prove our claim. First we show that $\mathcal{P}_{I'}^{g'}=\bigcup_{i}^r[c_i',d_i']$. Let $\alpha \in \mathcal{P}_{I'}^{g'}$. If $\alpha(1)\leq a$ then $\alpha \in \mathcal{P}_{I}^{g}$ that is $\alpha \in [c_i,d_i]$ for some $i$, because if $\alpha \notin \mathcal{P}_{I}^{g}$ then we have $x_1^{a+1}v_k|x^{\alpha}$ for some $k\leq r$ and therefore $\alpha(1)=a+1$, a contradiction.\\
\indent If $\alpha(1)<a$, it follows that $c_i(1)<a$ and therefore $c_i'=c_i.$ We get $c_i'=c_i\leq d_i\leq d_i'$, thus $\alpha \in [c_i',d_i'].$ If $\alpha(1)=a$, it follows that $d_i(1)=a$ and $d_i'(1)=a+1$. Suppose that $\alpha\notin [c_i',d_i'].$ Then $\alpha-e_1\in \mathcal{P}_{I}^{g}$  because $\alpha \in \mathcal{P}_{I'}^{g'}$, $\alpha(1)\neq a+1$ and thus $x_1^{a+1}v_k$ does not divide $x^{\alpha}$ for all $k\leq r$. It follows that $\alpha-e_1\in [c_j,d_j]$ for some $j\neq i$. But $(\alpha-e_1)(1)=a-1$ which implies $d_j(1)\geq a-1$. It follows that $\alpha \leq d_j+e_1=d_j'$ and therefore $\alpha \in [c_j',d_j']$, because $c_j'=c_j$.\\
\indent If $\alpha(1)=a+1$, it follows that $\alpha-e_1\in \mathcal{P}_I^g$ and therefore $\alpha-e_1\in [c_i,d_i]$ for some $i$. Indeed, if $x_1^{a+1}v_k|x^{\alpha}$ for some $k$ then $x_1^{a}v_k|x^{\alpha-e_1}$, else if $v_k|x^{\alpha}$ for some $k\geq r+1$ then $v_k|x^{\alpha-e_1}$. But $(\alpha-e_1)(1)=a$ and therefore $d_i(1)=a$. We get $\alpha\leq d_i+e_1=d_i'$ and thus $\alpha\in [c_i',d_i']$.\\

 Now, we must prove that for any $i\neq j$, we have $[c_i',d_i']\cap[c_j',d_j']=\emptyset$. Assume  that there exists some $\alpha \in [c_i',d_i']\cap[c_j',d_j']$. If $\alpha(1)<a$ then $\alpha \leq d_k$ and $\alpha\geq c_k'=c_k$ for $k=i,j$. Suppose that
  $\alpha(1)\geq a$. Then $d_j'(1)\geq a$. If $c_i(1)<a$ then $\alpha-e_1\geq c_i=c_i'$. If $c_i(1)=a_1$ then $\alpha -e_1\geq c_i'-e_1=c_i$. Hence
  $\alpha-e_1\in [c_i,d_i]$ and
  similarly, $\alpha-e_1\in [c_j,d_j]$, which gives again a contradiction.\\
\indent Now since $x_1\notin I'$ and $I':x_1=I$ then by \cite[Proposition 2]{D4} we have $\sdepth(I)\geq \sdepth(I').$ This completes the proof.
\end{proof}
\begin{Theorem}\label{th2}
Let $I=\bigcap\limits_{i=1}^kQ_i \subset S$ be a monomial ideal such that each $Q_i$ is irreducible and $G(\sqrt{Q_i})\cap G(\sqrt{Q_j})=\emptyset$ for all $i\neq j$, then $\sdepth(I)=\sdepth(\sqrt{I}).$
\end{Theorem}
\begin{proof}
We may suppose that $\height(Q_i)\geq 2$ for all $i$ because if for example $\height(Q_k)=1$, then we may remove $Q_k$ since $I\cong \cap_{i=1}^{k-1}Q_i$. Let $\{v_1,v_2,\dots,v_m\}$ be the set of minimal monomial generators of $I$. If $I$ is squarefree then $I=\sqrt{I}$. In the case that $I$ is not squarefree, we may assume that $x_1^2$ divides some $v_i$. We may further assume that $x_1|v_i$ for $i=1,\ldots,r$ and $x_1$ does not divide $v_i$ for $i>r$. Then, since $G(\sqrt{Q_i})\cap G(\sqrt{Q_j})=\emptyset$ for all $i\neq j$, it follows that there exists an integer $a\geq 1$ such that $x_1^{a+1}|v_i$ for $i=1,\ldots,r$. Now we apply Lemma 2.1 and conclude that $\sdepth(I) =\sdepth(v_1/x_1,...v_r/x_1,v_{r+1},...,v_m)$. Thus an obvious induction argument completes the proof.
%If $I$ is not squarefree then we may assume that $x_1^2|v_1,\dots,x_1^2|v_r$, where $r<m$. By Lemma \ref{lemma1}, $\sdepth(I)=\sdepth((v_1/x_1,\dots,v_r/x_1,v_{r+1},\dots,v_m)).$ Thus, we can replace $I$ with
%$(v_1/x_1,\dots,v_{r}/x_1,v_{r+1},\dots,v_m)$ and then we apply again the previous step if necessary. After a finite number of such  steps we get a square free ideal, that is  $\sqrt{I}$.
\end{proof}
\begin{Remark}{\em
In the setting of Theorem \ref{th2}, if $Q_i$ are not irreducible for all $i$ then the result is false. For example if $n=4$, $I=(x_1^2,x_1x_2,x_2^2)\cap (x_3^2,x_3x_4,x_4^2)$ and $\mathcal P$ is a partition of ${\mathcal P}^g_I$, $g=(2,2,2,2)$ then we must have 9 intervals $[a,b]$ in $\mathcal P$  starting with the generators $a$ of $I$ but only 8 monomials $b$ are in  ${\mathcal P}^g_I$ with $\rho(b)=3$, the biggest  one $x_1^2x_2^2x_3^2x_4^2$ cannot be taken. Thus $\sdepth I<3$. But clearly $\sdepth (\sqrt{I})=3.$
}
\end{Remark}
\begin{Lemma}\label{samedegree}
Let $I$ be a squarefree monomial ideal of $S$ generated by monomials of degree $d$. Let $A$ be the number of monomials of degree $d$ and $B$ be the number of monomials of degree $d+1$ in $I$. Then $$d\leq \sdepth(I)\leq d+\lfloor\frac{B}{A}\rfloor,$$
where $\lfloor a\rfloor$, $a\in \QQ$, denotes the largest integer which is not greater than $a$.
\end{Lemma}
\begin{proof}
We follow the proof of \cite[Theorem 2.8]{MI}(see also \cite{IQ} and \cite{KSSY}). Let $k:=\sdepth({I})$. The poset $\mathcal{P}_I^{g}$ has a partition $\mathcal{P}$ : $\mathcal{P}_I^{g}= \bigcup_{i=1}^{s}$ $[c_i,d_i]$, satisfying
$\sdepth(\mathcal{P})$\\$=k$. For each interval $[c_i,d_i]$ in $\mathcal{P}$ with $|c_i|=d$ we have $|d_i|\geq k$. Also there are $|d_i|-|c_i|$ subsets of cardinality $d+1$ in this interval. Since these intervals are disjoint, counting the number of subsets of cardinality $d$ and $d+1$ we have $(k-d)A\leq B$ that is $k\leq d+\frac{B}{A}$. Hence $\sdepth(I)\leq d+\lfloor\frac{B}{A}\rfloor.$ The other inequality follows by \cite[Lemma 2.1]{KSSY}.
\end{proof}
\begin{Corollary}
Let $I$ be a squarefree monomial ideal of $S$ generated by monomials of degree $d$. If $\binom{n}{d+1}<|G(I)|$ then $\sdepth(I)=d$.
\end{Corollary}

\begin{Theorem}\label{intersections}
Let $I=\bigcap\limits_{i=1}^kP_i$ be a monomial ideal in $S$ where each $P_i$ is a monomial prime ideal and $\sum\limits_{i=1}^k P_i=\mathfrak{m}$. Suppose that $G(P_i)\cap G(P_j)=\emptyset$ for all $i\neq j$. Then $$\sdepth(I)\leq \frac{n+k}{2}.$$
\end{Theorem}
\begin{proof}
Note that $I$ is generated by monomials of degree $k$. Let $|G(P_i)|=d_i$ then clearly the number of monomials of degree $k$ in $I$ is $d_1d_2\cdots d_k$. Now let $v\in I$ be a monomial of degree $k+1$, then $v$ is divisible by at least one variable from $G(P_{i})$ for all $1\leq i\leq k$ and one of these sets contains two variables from $\supp(v):=\{x_j:x_j|v\}$. We may suppose that all $d_i\geq 2$ because if for example $d_k=1$, then we may remove $P_k$ because $I\cong \cap_{i=1}^{k-1}P_i$. Fix first $G(P_1)$ such that $\supp(v)$ contains two variables from $G(P_1)$ then the number of such monomials is $\binom{d_1}{2}d_2d_3\cdots d_k$. Now let $\supp(v)$ contains two variables from $G(P_2)$ then the number of such monomials is $\binom{d_2}{2}d_1d_3\cdots d_k$.
Continuing in the same way for all $G(P_{i})$, $3\leq i\leq k$. We get the total number of monomials of degree $k+1$ is $$\sum\limits_{i=1}^k\binom{d_i}{2}d_1\cdots d_{i-1}d_{i+1}\cdots d_k=\frac{d_1d_2\cdots d_k}{2}(\sum\limits_{i=1}^kd_i-k)=\frac{d_1d_2\cdots d_k}{2}(n-k).$$
Now by Lemma \ref{samedegree} we have $$\sdepth(I)\leq k+\frac{1}{d_1d_2\cdots d_k}\cdot\frac{d_1d_2\cdots d_k}{2}(n-k)=k+\frac{n-k}{2}=\frac{n+k}{2}.$$
\end{proof}

Let $I=\bigcap\limits_{i=1}^kQ_i$ be a monomial ideal such that each $Q_i$ is irreducible and $G(\sqrt{Q_i})\cap G(\sqrt{Q_j})=\emptyset$ for all $i\neq j$, $\height(Q_i)=d_i$ and $\sum\limits_{i=1}^k\sqrt{Q_i}=\mathfrak{m}$. We define a set $$A:=\{Q_{i}\,:\,\,\,\height(Q_i)\hbox{ is odd }\}.$$
\begin{Corollary}\label{cor2}
Let $I=\bigcap\limits_{i=1}^kQ_i$ be a monomial ideal such that each $Q_i$ is irreducible and $G(\sqrt{Q_i})\cap G(\sqrt{Q_j})=\emptyset$ for all $i\neq j$, $\height(Q_i)=d_i$ and $\sum\limits_{i=1}^k\sqrt{Q_i}=\mathfrak{m}$. Then $$\frac{n+|A|}{2}\leq\sdepth(I)\leq \lfloor\frac{n+k}{2}\rfloor.$$
\end{Corollary}
\begin{proof}
By \cite[Lemma 1.2]{AD} we have $\lceil\frac{d_1}{2}\rceil+\lceil\frac{d_2}{2}\rceil+\dots+\lceil\frac{d_k}{2}\rceil\leq \sdepth(I)$. Now by Theorem \ref{th2} we have $\sdepth(I)=\sdepth(\sqrt{I})$ and by Theorem \ref{intersections} the required result follows.
\end{proof}
\begin{Remark}\label{r1}
{\em With the hypothesis from the above Corollary, $n$ is odd if and only if $|A|$ is odd, thus $n+|A|$ is always even.
}
\end{Remark}
\begin{Corollary}\label{cor3}
Let $I=\bigcap\limits_{i=1}^kQ_i$ be a monomial ideal such that each $Q_i$ is irreducible and $G(\sqrt{Q_i})\cap G(\sqrt{Q_j})=\emptyset$ for all $i\neq j$ and $\sum\limits_{i=1}^k\sqrt{Q_i}=\mathfrak{m}$. Suppose that $|A|=k$, then
$$\sdepth(I)=\frac{n+k}{2}.$$
\end{Corollary}
\begin{proof}
The proof follows by Corollary \ref{cor2} and Remark \ref{r1}.
\end{proof}
\begin{Corollary}\label{cor5}
Let $I=\bigcap\limits_{i=1}^kQ_i$ be a monomial ideal such that each $Q_i$ is irreducible and $G(\sqrt{Q_i})\cap G(\sqrt{Q_j})=\emptyset$ for all $i\neq j$ and $\sum\limits_{i=1}^k\sqrt{Q_i}=\mathfrak{m}$. Suppose that $k$ is odd and $|A|=k-1$ then, $$\sdepth(I)=\frac{n+k-1}{2}\,.$$
\end{Corollary}
\begin{proof}
Since $k$ is odd then $k-1$ is even and since $|A|=k-1$ is even thus $n$ is even and so $n+k-1$ is even. The result follows by Corollary \ref{cor2}.
\end{proof}

\begin{Remark} {\em The bounds found above for $\sdepth(I)$ are certainly sufficient to show Stanley's Conjecture in this case because $\sdepth(I)\geq k=\depth(I)$, as it is done in \cite[Theorem 1.4]{AD} and \cite[Theorem 3.2]{MI1} in a more general frame. It is a difficult combinatorial question to find the precise value of $\sdepth(I)$. It is hard to find $\sdepth(I)$ for example in the case $k=2$ and $n$ even(see \cite[Corollary 2.10]{MI}). We believe that in this case $\sdepth(I)= (n/2)+1$ but we can prove it only for small $n$. The trouble is that we must construct a partition $\mathcal P$ on ${\mathcal P}_I^g$ such that $\sdepth({\mathcal P})=(n/2)+1$. For example if $n=6$ and $I=(x_1,x_2)\cap
(x_3,\ldots,x_6)$ then we have the partition ${\mathcal P}_I^g=[13,1345]\cup [14,1456]\cup [15,1256]\cup [16,1236]\cup [23,1234]\cup [24, 1245]
\cup [25, 2356]\cup [26,2346]\cup (\cup_C [C,C])$, where $C\subset [6]$ with $|C|= 5,6$, or $C\in \{1235, 1246, 1346, 1356, 2345, 2456\}.$  Thus $\sdepth(I)=4=(n/2)+1$.
}
\end{Remark}

\begin{Corollary}
Let $I$ be the edge ideal of a complete $k$-partite hypergraph $\mathbf{H}_d^k$. Then
$$\sdepth(I)=\frac{n+k}{2}, \,\,\,\,\,\,\,\,\,\,\,\,\,\,\,\,\,\,\,\, if\,\,d\,\,is\,\,odd;$$
$$\frac{n}{2}\leq \sdepth(I)\leq \frac{n+k}{2},\,\,\,\,\,\,if\,\,d\,\,is\,\,even.$$
\end{Corollary}
In the next theorem we give an upper bound for the Stanley depth of any monomial ideal of $S$.
We will need the following lemma.
\begin{Lemma}(\cite[Lemma 4.4]{MI1}) \label{is} Let $I\subset {\hat S} = S[x_{n+1}]$ be a monomial ideal, $x_{n+1}$ being a new variable. If
$I\cap S \not= (0)$, then $\sdepth_S(I\cap S) \geq \sdepth_{\hat S} I'- 1.$
\end{Lemma}
\begin{Theorem}\label{General}
Let $I$ be a monomial ideal and let $\Min(S/I)=\{P_1,\dots,P_s\}$ with $\sum\limits_{i=1}^sP_i=\mathfrak{m}$. Let $d_{i}:=|G(P_i)\backslash G(\sum\limits_{i\neq j}^{s}P_j)|$, and $r:=|\{d_i:d_i\neq 0\}|$. Suppose that $r\geq 1$. Then $$\sdepth(I)\leq (2n+r-\sum_{i=1}^sd_i)/2.$$
\end{Theorem}
\begin{proof}
We may assume that $P_{i}'$'s be prime ideals such that $G(P_{i}')=G(P_{i})\backslash G(\sum\limits_{i\neq j}^{s}P_j)$ and $\height(P_{i}')\neq 0$ for all $1\leq i\leq r\leq s$. We have $G(P_{i}')\cap G(P_{j}')=\emptyset$ for all $i\neq j$. Let $A:=\bigcup_{i=1}^rG(P_{i}')$, $S'=K[A]$ and $B:=\{x_1,\dots,x_n\}\backslash \bigcup_{i=1}^rG(P_{i}')$, then $|B|=n-\sum_{i=1}^rd_{i}$. Now applying Lemma \ref{is} by recurrence on the set $B$ we have $$\sdepth_S(\sqrt{I})\leq \sdepth_{S'}\Big(\bigcap\limits_{i=1}^rP_{i}'\Big)+|B|.$$ By Theorem \ref{intersections} we have $\sdepth_{S'}\Big(\bigcap\limits_{i=1}^rP_{i}'\Big)\leq \frac{|A|+r}{2}$. But $|A|=n-|B|$ thus we have $$\sdepth(\sqrt{I})\leq \frac{n-|B|+r}{2}+|B|=\frac{2n+r-\sum_{i=1}^sd_i}{2}.$$ By  \cite[Corollary 2.2]{MI} $\sdepth(I)\leq \sdepth(\sqrt{I})$ and  we get $\sdepth(I)\leq (2n+r-\sum_{i=1}^sd_i)/2$.
\end{proof}
\begin{Example}
{\em
Let $I=(x_1\dots,x_9)\cap(x_9,\dots,x_{18})\cap(x_{18},\dots,x_{27})\cap(x_{27},\dots,x_{36})\subset K[x_1,\dots,x_{36}]$. We have $d_1=8$, $d_2=8$, $d_3=8$, $d_4=9$ and $s=4$, then by Theorem \ref{General} we have $\sdepth(I)\leq 23$.
}
\end{Example}
\begin{Remark}
{\em In the above example we have that $\sdepth(I)\leq 23$ but by \cite[Theorem 2.1]{MI1} we have only $\sdepth(I)\leq 31$.
}
\end{Remark}
\bigskip
\section{Cyclic modules associated to a complete $k$-partite hypergraph}
\begin{Theorem}\label{thm}
Let $S=K[x_1,\ldots,x_n]$ be a polynomial ring and $Q_1,Q_2,\ldots,Q_k$ monomial irreducible ideals of $S$ such that $G(\sqrt{Q_i})\cap G(\sqrt{Q_j})=\emptyset$ for all $i\neq j$. Let $r_i:=\height(Q_i)$, $\sum\limits_{i=1}^{k}r_i=n$. If $I=Q_1\cap Q_2\cap\ldots\cap Q_k$, then
\begin{multline*}
\sdepth(S/I)\geq \min\Big\{n-r_1,\min\limits_{2\leq i\leq k}\{\lceil\frac{r_1}{2}\rceil+\ldots+\lceil\frac{r_{i-1}}{2}\rceil+r_{i+1}+\ldots+
r_k\}\Big\},
\end{multline*}
where $\lceil a\rceil$, $a\in \QQ$, denotes the smallest integer which is not less than $a$.
\end{Theorem}
\begin{proof}
 As a $K$-linear space $S/I$ is isomorphic to the direct sum of some multigraded modules as,
\begin{multline*}
S/I\cong S/Q_1 \oplus (Q_1/Q_1\cap Q_2)\oplus (Q_1\cap Q_2/Q_1\cap Q_2\cap Q_3)\oplus\ldots\\\oplus (Q_1\cap\ldots\cap Q_{k-1}/Q_1\cap\ldots\cap Q_k),
\end{multline*}
and we have\\
$\sdepth(S/I)\geq$ $$\min\{\sdepth S/Q_1,\sdepth Q_1/(Q_1\cap Q_2),\ldots,\sdepth(Q_1\cap\ldots\cap Q_{k-1})/(Q_1\cap\ldots\cap Q_k)\}.$$
By \cite[Lemma 1.1]{DQ} $\sdepth(S/Q_1)=n-r_1$. Now since
$$(Q_1\cap\ldots\cap Q_{i-1})/(Q_1\cap\ldots\cap Q_i)\cong(Q_1\cap\ldots\cap Q_{i-1}+Q_i)/Q_i$$
and
$$(Q_1\cap\ldots\cap Q_{i-1}+Q_i)/Q_i\cong Q_1\cap\ldots\cap Q_{i-1}\cap K[x_j \mid x_j\not\in G(\sqrt{Q_i})]$$
Now by using \cite[Lemma 1.2]{AD} and \cite[Lemma 3.6]{HVZ} we have
$$\sdepth((Q_1\cap\ldots\cap Q_{i-1})/(Q_1\cap\ldots\cap Q_i))\geq \lceil\frac{r_1}{2}\rceil+\ldots+\lceil\frac{r_{i-1}}{2}\rceil+r_{i+1}+\ldots+r_k.$$
Thus $$\min\{\sdepth S/Q_1,\sdepth Q_1/(Q_1\cap Q_2),\ldots,\sdepth(Q_1\cap\ldots\cap Q_{k-1})/(Q_1\cap\ldots\cap Q_k)\}$$
\[\geq \min\Big\{n-r_1,\min\limits_{2\leq i\leq k}\{\lceil\frac{r_1}{2}\rceil+\ldots+\lceil\frac{r_{i-1}}{2}\rceil+r_{i+1}+\dots+r_k\}\Big\}.\]
\end{proof}
\begin{Remark}{\em The  theorem says in particular that Stanley's conjecture holds for $S/I$ in the above settings, because $\sdepth(I)\geq k-1=\depth(S/I)$ as it was noticed in \cite[Theorem 3.2]{MI1}. The found lower bound depends on the numbering of $(r_i)$ as shows the following example.
}
\end{Remark}
\begin{Example}{\em
Let $I=P_1\cap P_2\cap P_3\cap P_4\subset K[x_1,\dots,x_{14}]$ where $G(P_i)\cap G(P_j)\neq \emptyset$ for all $i\neq j$, $P_1+P_2+P_3+P_4=(x_1,\dots,x_{14})$ and $\height(P_1)=5$, $\height(P_2)=4$, $\height(P_3)=3$, $\height(P_4)=2$. Then by above theorem we have $\sdepth(S/I)\geq \min\{9,\min\{8,7,7\}\}=7$. Now let us reorder these primes or equivalently consider that $\height(P_1)=4$, $\height(P_2)=5$, $\height(P_3)=2$, $\height(P_4)=3$. Then again by the above theorem we have $\sdepth(S/I)\geq \min\{10,\min\{7,8,6\}\}=6$.
}
\end{Example}
\begin{Corollary}\label{clb}
Let $r_1\geq r_2=\ldots=r_k$, then
$$\sdepth(S/I)\geq \lceil\frac{r_1}{2}\rceil+\ldots+\lceil\frac{r_{k-1}}{2}\rceil.$$
\end{Corollary}
\begin{Corollary}\label{clc}
Let $I$ be the edge ideal of a complete $k$-partite hypergraph $\mathbf{H}_d^k$. Then $$\sdepth(S/I)\geq \lceil\frac{d}{2}\rceil(k-1).$$
\end{Corollary}
\begin{Remark}{\em
In Corollary \ref{clc} if we take $k=2$ then we have $\sdepth(S/I)\geq \lceil\frac{d}{2}\rceil$ but by \cite{DQ} and \cite{BK} we have $\sdepth(S/I)=\lceil\frac{d}{2}\rceil$. This shows that the bound is equal to the actual value in this case. If $d$ is odd in Corollary \ref{clc} then $\sdepth(S/I)\geq \frac{n+k}{2}-\frac{d+1}{2}$. Now assume that $\sdepth(I)\geq \sdepth(S/I)+1$, as A. Rauf asks in \cite{R1}, (for $k=2$ this inequality is true \cite{DQ}). Then by Theorem \ref{intersections} we have $\sdepth(S/I)\leq \frac{n+k}{2}-1$. If this is the case then our lower bound by Corollary \ref{clc} could be a reasonable one, as it is clear that for $d=3$ we have $\frac{n+k}{2}-2\leq \sdepth(S/I)\leq\frac{n+k}{2}-1$.
}
\end{Remark}
\indent Next we give an upper bound for the Stanley depth of $S/I$ with $I=P_1\cap P_2\cap \cdots\cap P_k$, ($P_i$) being monomial prime ideals such that $G(P_i)\cap G(P_j)=\emptyset$. Let $\height(P_i)=r_i$. By \cite[Lemma 3.6]{HVZ} it is enough to consider that $P_1+P_2+\dots+P_k=\mathfrak{m}.$ Let $\mathcal{D}:S/I=\bigoplus\limits_{i=1}^pu_iK[Z_i]$ be a Stanley decomposition. Then $Z_i$ cannot have in the same time variables from all $G(P_i)$, otherwise $u_iK[Z_i]$ will not be a free $K[Z_i]$-module. Suppose $u_1=1$ and $Z_1\subset \{x_{r_1+1},\dots,x_{n}\}$.
\begin{Lemma}\label{ub}
Let $r_2\geq r_3\geq \dots \geq r_{k}\geq 1$. Then $$\sdepth(\mathcal{D})\leq \lceil\frac{r_1}{2}\rceil+r_2+r_3+\cdots +r_{k-1}.$$
\end{Lemma}
\begin{proof}
We can assume that $P_1=(x_1,\dots,x_{r_1})$ and $$\psi:P_1\cap K[x_{1},\dots,x_{r_1}]\hookrightarrow S/{I}$$ be the inclusion given by $$K[x_{1},\dots,x_{r_1}]\hookrightarrow S/{I}.$$ Then $P_1\cap K[x_1,\dots,x_{r_1}]=\bigoplus\limits_i\psi^{-1}(u_iK[Z_i])$. If $\psi^{-1}(u_iK[Z_i])\neq 0$ implies there exists $u_if\in u_iK[Z_i]$ with $u_if\in P_1\cap K[x_{1},\dots,x_{r_1}]$, but then $u_i\neq1$ and so $u_i\in P_1\cap K[x_{1},\dots,x_{r_1}]$. Let $Z_i'=Z_i\cap \{x_{1},\dots,x_{r_1}\}$. Then $\psi^{-1}(u_iK[Z_i])=u_iK[Z_i']$ and we get a Stanley decomposition of $J:=P_1\cap K[x_{1},\dots,x_{r_1}]$. Since $\sdepth(J)=\lceil\frac{r_1}{2}\rceil$ by \cite{B} it follows that $$|Z_i'|\leq \sdepth(J)=\lceil\frac{r_1}{2}\rceil.$$ But $Z_i$ cannot have variables from all $G(P_j)$ and we get $Z_i\subset \bigcup\limits_{1=e\neq j}^kG(P_e)$ for some $j\neq 1$. Therefore $$Z_i\subset \{Z_i'\cup G(P_{2})\cup G(P_{3})\cup \cdots \cup G(P_{k})\}\backslash G(P_j)\text{ for some $j\neq 1$}.$$ Thus we have $$\sdepth(\mathcal{D})\leq \lceil\frac{r_1}{2}\rceil+r_2+r_3+\dots+r_{j-1}+r_{j+1}+\dots+r_{k}\leq \lceil\frac{r_1}{2}\rceil+r_2+r_3+\cdots +r_{k-1}.$$
\end{proof}
\begin{Proposition}\label{prop}
Suppose that $r_1\geq r_2\geq\dots \geq r_k$, $k\geq 3$. Then $$\sdepth(S/I)\leq \lceil\frac{r_{k-1}}{2}\rceil+r_1+r_2+\dots+r_{k-2}.$$
\end{Proposition}
\begin{proof}
Let $\mathcal{D}$ be a Stanley decomposition of $S/I$, $S/I=\bigoplus\limits_{i=1}^pu_iK[Z_i]$ such that $\sdepth(\mathcal{D})=\sdepth(S/I)$. Then there exists $1\leq m\leq k$ such that $Z_1\subset \bigcup\limits_{1=e\neq m }^kG(P_e)$. By the above lemma $\sdepth(\mathcal{D})\leq \sum\limits_{1=i\neq m}^{k-1}r_i+\lceil\frac{r_m}{2}\rceil$, which is enough.
\end{proof}
\begin{Corollary}
Let $I$ be the edge ideal of a complete $k$-partite hypergraph $\mathbf{H}_d^k$. Then $$(k-1)\lceil\frac{d}{2}\rceil\leq\sdepth(S/I)\leq (k-2)d+\lceil\frac{d}{2}\rceil.$$
\end{Corollary}

\end{document}